\documentclass[12pt, eqno]{article}

\usepackage{amsmath,amssymb,latexsym}

\usepackage{amssymb,latexsym}
\usepackage{amsmath,latexsym}
\usepackage[top=1in,bottom=1in,left=1in,right=1in]{geometry}
\usepackage{amscd}
\usepackage{amsthm}

\newtheorem{theorem}{Theorem}[section]

\newtheorem{corollary}[theorem]{Corollary}

\newtheorem{definition}[theorem]{Definition}

\newtheorem{lemma}[theorem]{Lemma}

\begin{document}

\title{Quasi-Einstein structures and almost cosymplectic manifolds}

\author{\\  Xiaomin Chen
\thanks{
The author is supported by Natural Science Foundation of Beijing, China (Grant No.1194025).
 }\\
{\normalsize College of Science, China University of Petroleum (Beijing),}\\
{\normalsize Beijing, 102249, China}\\
{\normalsize xmchen@cup.edu.cn}}
\maketitle \vspace{-0.1in}


\abstract{In this article, we study almost cosymplectic manifolds admitting quasi-Einstein structures $(g, V, m, \lambda)$.
First we prove that an almost cosymplectic $(\kappa,\mu)$-manifold is locally isomorphic
to a Lie group if $(g, V, m, \lambda)$ is closed and on a compact almost $(\kappa,\mu)$-cosymplectic manifold there do not exist
quasi-Einstein structures $(g, V, m, \lambda)$, in which the potential vector field $V$ is collinear with the Reeb vector filed $\xi$. Next we consider an  almost $\alpha$-cosymplectic manifold admitting a quasi-Einstein structure and obtain some results. Finally, for a $K$-cosymplectic manifold with a closed, non-steady quasi-Einstein structure, we prove that it is $\eta$-Einstein. If $(g, V, m, \lambda)$ is non-steady and $V$ is a conformal vector field, we obtain the same conclusion.}
 \vspace{-0.1in}
\medskip\vspace{12mm}

\noindent{\it Keywords}:  quasi-Einstein structures; almost cosymplectic $(\kappa,\mu)$-manifolds; almost $\alpha$-cosymplectic manifolds; cosymplectic manifolds; Einstein manifolds.
  \vspace{2mm}

\noindent{\it MSC}: 53C25; 53D15 \vspace{2mm}

\section{Introduction}

Einstein metrics is an important for both mathematics and physics. But the pure Einstein theory is always  too strong as a system model for various physical questions, thus its several generalizations have been
studied. A Ricci soliton is a Riemannian metric, which satisfies
\begin{equation*}
 \frac{1}{2}\mathcal{L}_V g+Ric-\lambda g=0,
\end{equation*}
where $V$ and $\lambda$ are the potential vector field and some constant, respectively. It is clear that a trivial Ricci soliton is an Einstein metric with $V$ zero or Killing. When the potential vector field $V$ is a gradient vector field, i.e. $V=Df$, where $f$ is a smooth function, then it is called a \emph{gradient Ricci soliton.}

 An interesting generalization of Einstein metrics was proposed by Case \cite{C}, which arises from the $m$-Bakry-Emery Ricci tensor. The Ricci tensor is defined as follows:
 \begin{equation*}
   \mathrm{Ric}^m_f=\mathrm{Ric}+\nabla^2f-\frac{1}{m}df\otimes df,
    \end{equation*}
   where the integer $m$ satisfies $0<m\leq\infty$, $\nabla^2f$ denotes the Hessian form of the smooth function $f$. We call a triple $(g, f, m,\lambda)$ (a Riemannian manifold $(M, g)$ with a function
$f$ on $M$) ($m$-)\emph{quasi-Einstein structure} if it satisfies the equation
\begin{equation}\label{1}
{\rm Ric}+\nabla^2f-\frac{1}{m}df\otimes df = \lambda g
\end{equation}
for some $\lambda\in\mathbb{R}$. Notice that Equation \eqref{1} recovers the gradient Ricci soliton when $m=\infty$. A quasi-Einstein metric is an Einstein metric if $f$ is constant.
For a general manifold,  quasi-Einstein metrics have been studied in depth and some rigid properties and gap results were obtained (cf.\cite{CSW,W,W2}).

Later on Barros-Ribeiro Jr \cite{BR} and Limoncu \cite{L} generalized and studied the previous equation \eqref{1},
independently, by considering a 1-form $V^\flat$ instead of $df$ , which is satisfied
\begin{equation}\label{4.27}
\mathrm{Ric}+\frac{1}{2}\mathcal{L}_V g-\frac{1}{m}V^\flat\otimes V^\flat = \lambda g,
\end{equation}
where $V^\flat$ is the 1-form associated to $V$. In particular, if the 1-form $V^\flat$ is closed, we call quasi-Einstein structure $(g, V, m, \lambda)$ is \emph{closed}. Using the terminology of Ricci solitons, we call a quasi-Einstein
structure \emph{shrinking, steady or expanding}, respectively, if $\lambda<0, \lambda=0$, or $\lambda>0.$
When $V\equiv0$, a quasi-Einstein structure is said to be \emph{trivial} and in this case, the metric becomes an Einstein metric.
It is mentioned that a quasi-Einstein structure $(g, V,m, \lambda)$ is reduced to a Ricci soliton when $m=\infty.$

Recently, Barros-Gomes in \cite{BG3} further studied Eq.\eqref{4.27} and they proved that if a compact quasi-Einstein $(M^n ,g,V,m,\lambda)$, $n\geq3$ is Einstein, then $V$ vanishes identically.  On the other hand, we also notice that for the odd-dimensional manifold, Ghosh in \cite{GH2} studied  contact metric manifolds with quasi-Einstein structure $(g, f,m, \lambda)$.
More recently, he considered quasi-Einstein structure $(g, V, m, \lambda)$ in the framework of contact metric manifolds (see \cite{GH1}).

Remark that another class of almost contact manifold, called \emph{almost cosymplectic manifold}, was also paid many attentions (see a survey \cite{MNY}). The concept was first defined by Goldberg and Yano \cite{GY} as an almost contact manifold whose 1-form $\eta$ and fundamental 2-form $\omega$ are closed.  An almost cosymplectic manifold is said to be \emph{cosymplectic} if in addition the almost contact structure is normal (notice that here we adopt  "cosymplectic" to represent "coK\"ahler" in \cite{MNY}).
Recently, Bazzoni-Goertsches \cite{BG} defined a \emph{$K$-cosymplectic manifold}, namely an almost cosymplectic manifold whose Reeb vector field is Killing. In \cite{CP}, in fact it is proved that every compact Einstein $K$-cosymplectic manifold is necessarily cosymplectic.
In addition, Endo \cite{E} defined the notion of \emph{almost cosymplectic $(\kappa,\mu)$-manifold}, i.e. the
curvature tensor of an almost cosymplectic manifold satisfies
\begin{equation}\label{4.1}
  R(X,Y)\xi=\kappa(\eta(Y)X-\eta(X)Y)+\mu(\eta(Y)hX-\eta(X)hY)
\end{equation}
for any vector fields $X,Y$, where $\kappa,\mu$ are constant and $h=\frac{1}{2}\mathcal{L}_\xi\phi$.  As the extension of almost cosymplectic manifold,  Kenmotsu \cite{K} defined the {\it almost Kenmotsu manifold}, which is an almost contact manifold satisfying $d\eta=0$ and $d\omega=2\eta\wedge\omega$.
    Based on this Kim and Pak \cite{KP} introduced the concept of \emph{almost $\alpha$-cosymplectic manifold}, i.e. an almost contact manifold satisfying
$d\eta=0$ and $d\omega=2\alpha\eta\wedge\omega$ for some real number $\alpha$.

Motivated by the above background, in the present paper we mainly consider three classes of almost cosymplectic manifolds with quasi-Einstein structures $(g, V, m, \lambda)$ including almost $(\kappa,\mu)$-manifolds, almost $\alpha$-cosymplectic manifolds and $K$-cosymplectic manifolds, i.e.  almost cosymplectic manifolds with Killing Reeb vector field $\xi$.  In order to prove our results, we need to recall some definitions and related conclusions on almost cosymplectic manifolds as well as quasi-Einstein structures, which are presented in Section 2. Starting from Section 3, we will state our results and give their proofs.

\section{Preliminaries}
 Let $M^{2n+1}$ be a $(2n+1)$-dimensional smooth manifold.
An \emph{almost contact structure} on $M$ is a triple $(\phi,\xi,\eta)$, where $\phi$ is a
$(1,1)$-tensor field, $\xi$ a unit vector field, called Reeb vector field, $\eta$ a one-form dual to $\xi$ satisfying
$\phi^2=-I+\eta\otimes\xi,\,\eta\circ \phi=0,\,\phi\circ\xi=0.$
A smooth manifold with such a structure is called an \emph{almost contact manifold}.

A Riemannian metric $g$ on $M$ is called compatible with the almost contact structure if
\begin{equation*}
g(\phi X,\phi Y)=g(X,Y)-\eta(X)\eta(Y),\quad g(X,\xi)=\eta(X)
\end{equation*}
for any $X,Y\in\mathfrak{X}(M)$. An almost contact structure together with a compatible metric
is called an \emph{almost contact metric structure} and $(M,\phi,\xi,\eta,g)$ is called an almost contact metric manifold. An almost contact structure $(\phi,\xi,\eta)$ is said
to be \emph{normal} if the corresponding complex structure $J$ on $M\times\mathbb{R}$ is integrable.

Denote by $\omega$ the fundamental 2-form on $M$ defined by $\omega(X,Y):=g(\phi X,Y)$ for all $X,Y\in\mathfrak{X}(M)$.
An {\it almost $\alpha$-cosymplectic manifold} (\cite{KP,OAM}) is an almost contact metric manifold $(M,\phi,\xi,\eta,g)$ such that the fundamental form $\omega$ and 1-form $\eta$ satisfy $d\eta=0$ and $d\omega=2\alpha\eta\wedge\omega,$ where $\alpha$ is a real number. A normal almost $\alpha$-cosymplectic manifold is called an \emph{$\alpha$-cosymplectic manifold}.  $M$ is an {\it almost cosymplectic manifold} if $\alpha=0$.

Let $M$ be an almost $\alpha$-cosymplectic manifold, we recall that there is an operator
$h=\frac{1}{2}\mathcal{L}_\xi\phi$ which is a self-dual operator. In particular, if $h=0$, $M$ is normal. The Levi-Civita connection
is given by (see \cite{OAM})
\begin{equation}\label{2.4*}
  2g((\nabla_X\phi)Y,Z)=2\alpha g(g(\phi X,Y)\xi-\eta(Y)\phi X,Z)+g(N(Y,Z),\phi X)
\end{equation}
for arbitrary vector fields $X,Y$, where $N$ is the Nijenhuis torsion of $M$.  Then by a simple calculation, we have
\begin{equation}\label{2.2*}
\mathrm{trace}(h)=0,\quad h\xi=0,\quad\phi h=-h\phi,\quad g(hX,Y)=g(X,hY),\quad\forall X,Y\in\mathfrak{X}(M).
\end{equation}

 Using \eqref{2.4*}, a straightforward calculation gives
\begin{equation}\label{2.5}
\nabla_X\xi= -\alpha\phi^2X-\phi hX
\end{equation}
and $\nabla_\xi\phi=0$. Denote by $R$ and $\mathrm{Ric}$ the Riemannian curvature tensor and Ricci tensor, respectively. For an almost $\alpha$-cosymplectic manifold $(M^{2n+1},\phi,\xi,\eta,g)$ the following equations were proved(\cite{OAM}):
\begin{align}
&R(X,\xi)\xi-\phi R(\phi X,\xi)\xi=2[\alpha^2\phi^2X-h^2X]\label{2.6},\\
&\mathrm{trace}(\phi h)=0,\label{2.9}\\
 &R(X,\xi)\xi=\alpha^2\phi^2X+2\alpha\phi hX-h^2X+\phi(\nabla_\xi h)X\label{2.10}
 \end{align}
for any vector fields $X,Y$ on $M$.

Next we recall two important lemmas for a Riemannian manifold satisfying quasi-Einstein equation \eqref{1}.
\begin{lemma}[\cite{GH2}]\label{L3.2}
For a quasi-Einstein $(M,g, f,m, \lambda)$, the curvature tensor $R$ can be expressed as
\begin{align*}
R(X,Y )D f =&(\nabla_YQ)X-(\nabla_XQ)Y-\frac{\lambda}{m}\{X(f)Y-Y(f)X\}\\
&+\frac{1}{m}\{X(f)QY-Y(f)QX\}
\end{align*}
for any vector fields $X,Y$ on $M$, where $Q$ is the Ricci operator of $M$.
\end{lemma}
\begin{lemma}[\cite{CSW}]
For a quasi-Einstein $(M^{2n+1}, g, f, m,\lambda)$, the following equations hold:
\begin{align}
\frac{1}{2}Dr =& \frac{m-1}{m}Q(D f)+\frac{1}{m}\Big(r-2n\lambda\Big)Df\label{2.13},\\
  \frac{1}{2}\Delta r-\frac{m+2}{2m}g(Df,Dr)=& -\frac{m-1}{m}\Big|{\rm Ric}-\frac{r}{2n-1}g\Big|^2\label{3.10*}\\
  &-\frac{m+2n}{m(2n+1)}\Big(r-(2n+1)\lambda\Big)\Big(r-\frac{2n(2n+1)}{m+2n}\lambda\Big).\nonumber
\end{align}
Here $r$ denotes the scalar curvature of $M$.
\end{lemma}
In the following sections we always suppose that $(M^{2n+1},\phi,\xi,\eta,g)$ is an almost cosymplectic manifold and $g$ represents a quasi-Einstein metric.

\section{Almost cosymplectic $(\kappa,\mu)$-manifolds}
In this section we suppose that $(M^{2n+1},\phi,\xi,\eta,g)$ is an almost cosymplectic $(\kappa,\mu)$-manifold, namely the
curvature tensor satisfies \eqref{4.1}. By definition, Eqs.\eqref{2.2*}-\eqref{2.10} with $\alpha=0$ hold. Furthermore, the following relations are provided (see \cite[Eq.(3.22) and Eq.(3.23)]{MNY}):
\begin{align}
  Q=&2n\kappa\eta\otimes\xi+\mu h,\label{3.1**}\\
  h^2=&\kappa\phi^2\label{3.2}.
\end{align}

Using \eqref{2.2*}, it follows from \eqref{3.1**} that the scalar curvature $r=2n\kappa$ and $Q\xi=2n\kappa\xi.$
By \eqref{3.2}, we find easily that $\kappa\leq0$ and $\kappa=0$ if and only if $M$ is a cosymplectic manifold, thus in the following we always suppose
$\kappa<0$. Moreover, if $\mu=0$ the following conclusion was given.
\begin{theorem}\emph{(\cite[Theorem 4]{D})}\label{T2.1}
 An almost cosymplectic $(\kappa,0)$-manifold for some $\kappa<0$ is locally isomorphic
to a Lie group $G_\rho$ endowed with the almost cosymplectic structure, where
$\rho=\sqrt{-\kappa}$.
\end{theorem}
Making use of the above theorem we can prove the following conclusion.
\begin{theorem}\label{T1}
A $(2n+1)$-dimensional almost
cosymplectic $(\kappa,\mu)$-manifold with $\kappa<0$, admitting a closed quasi-Einstein structure $(g, V, m, \lambda)$, is locally isomorphic
to the above Lie group $G_\rho$. Moreover, either $\lambda=0$ or $\lambda=(2n+m)\kappa$.
\end{theorem}
\begin{proof}

In view of \eqref{3.1**} and Eq.\eqref{4.27}, we obtain
\begin{equation}\label{3.3}
  \nabla_YV=\lambda Y-\mu hY-2n\kappa\eta(Y)\xi+\frac{1}{m}g(V,Y)V
\end{equation}
for any vector $Y$.
Using this we compute
\begin{align*}
  R(X,Y)V =& \nabla_X\nabla_YV-\nabla_Y\nabla_XV-\nabla_{[X,Y]}V \\
  = & -\mu (\nabla_Xh)Y-2n\kappa(\nabla_X\eta)(Y)\xi-2n\kappa\eta(Y)\nabla_X\xi\\
  &+\frac{1}{m}g(\nabla_XV,Y)V+\frac{1}{m}g(V,Y)\nabla_XV\\
  &+\mu (\nabla_Yh)X+2n\kappa(\nabla_Y\eta)(X)\xi+2n\kappa\eta(X)\nabla_Y\xi\\
  &-\frac{1}{m}g(\nabla_YV,X)V-\frac{1}{m}g(V,X)\nabla_YV.
\end{align*}
Taking an inn product of the above formula with $\xi$ and using \eqref{4.1}, we have
\begin{align}\label{3.4*}
  &-\kappa(\eta(Y)g(X,V)-\eta(X)g(Y,V))-\mu(\eta(Y)g(hX,V)-\eta(X)g(hY,V))\\
  = & -2\mu g(\phi Y, h^2X)+\frac{1}{m}g(\nabla_XV,Y)g(V,\xi)+\frac{1}{m}g(V,Y)g(\nabla_XV,\xi)\nonumber\\
  &-\frac{1}{m}g(\nabla_YV,X)g(V,\xi)-\frac{1}{m}g(V,X)g(\nabla_YV,\xi)\nonumber\\
  =&-2\mu g(\phi Y, h^2X)+\frac{\lambda-2n\kappa}{m}[g(V,Y)\eta(X)-g(V,X)\eta(Y)].\nonumber
\end{align}
Now replacing $X$ and $Y$ by $\phi X$ and $\phi Y$, respectively, yields
\begin{equation*}
\mu g(Y, h^2\phi X)=0
\end{equation*}
for any vector fields $X,Y$, which implies $\mu=0$. Further, it follows from \eqref{3.4*} that
\begin{equation*}
\frac{\lambda-2n\kappa-m\kappa}{m}[g(V,Y)\eta(X)-g(V,X)\eta(Y)]=0.
\end{equation*}
Putting $Y=\xi$ in the foregoing equation shows either $V=\eta(V)\xi$ or $\lambda=(2n+m)\kappa$.

If $V=\eta(V)\xi$, by \eqref{2.5}, Eq.\eqref{3.3} becomes
\begin{equation}\label{3.17}
  Y(F)\xi-F\phi hY=\lambda Y-\mu hY-2n\kappa\eta(Y)\xi+\frac{1}{m}F^2\eta(Y)\xi,
\end{equation}
where $F=\eta(V)$. Contracting \eqref{3.17} over $Y$ yields
\begin{equation*}
  \xi(F)=(2n+1)\lambda-2n\kappa+\frac{1}{m}F^2
\end{equation*}
since ${\rm trace}(h)=trace(\phi h)=0$. On the other hand,  choosing $Y=\xi$ in \eqref{3.17} we get
\begin{equation*}
  \xi(F)=\lambda -2n\kappa+\frac{1}{m}F^2.
\end{equation*}
Consequently, the preceding two equations imply $\lambda=0.$
\end{proof}
Because $\kappa<0$, by Theorem \ref{T1} the following conclusion is obvious.
\begin{corollary}
There do not exist expanding, closed quasi-Einstein structures $(g, V, m, \lambda)$ on an almost
cosymplectic $(\kappa,\mu)$-manifold with $\kappa<0$.
\end{corollary}
When $V=Df$, it is clear that $V^\flat$ is closed, thus we have
\begin{corollary}
An almost
$(\kappa,\mu)$-cosymplectic manifold with $\kappa<0$, admitting a quasi-Einstein structure $(g, f, m, \lambda)$, is locally isomorphic
to the above Lie group $G_\rho$.
\end{corollary}
Next we consider the potential vector field $V$ being collinear with Reeb vector field $\xi$ and prove the following non-existence.
\begin{theorem}
There do not exist quasi-Einstein structures $(g, V, m, \lambda)$ with $V=\eta(V)\xi$ on a compact almost $(\kappa,\mu)$-cosymplectic manifold with $\kappa<0$.
\end{theorem}
\begin{proof}
Suppose $V=F\xi$ for some function $F$. Differentiating this along any vector field $Y$ and using \eqref{2.5}, we get
\begin{equation}\label{4.18}
\nabla_YV=Y(F)\xi-F\phi hY.
\end{equation}
By \eqref{3.1**} and \eqref{4.18}, \eqref{4.27} becomes
\begin{align}\label{4.30}
  &\Big(2n\kappa-\frac{F^2}{m}\Big)\eta(X)\eta(Y)+\mu g(hX,Y)\\
  &+\frac{1}{2}[Y(F)\eta(X)+X(F)\eta(Y)] =\lambda g(X,Y).\nonumber
\end{align}
Replacing $X$ and $Y$ by $\phi X$ and $\phi Y$, respectively, we find
\begin{equation*}
-\mu g(hX,Y)=\lambda g(\phi X,\phi Y).
\end{equation*}
Letting $X=Y$ and contracting $X$ gives $\lambda=0$ since $\mathrm{trace}(h)=0$. Thus, by taking $Y=\xi$ in \eqref{4.30}, we derive
\begin{equation*}
  \xi(F)=-\Big(2n\kappa-\frac{F^2}{m}\Big).
\end{equation*}
Moreover, using \eqref{2.9} we derive from \eqref{4.18} that
\begin{equation*}
  0=\int_M{\rm div}VdM=\int_M\xi(F)dM=-\int_M\Big(2n\kappa-\frac{F^2}{m}\Big)dM,
\end{equation*}
where $dM$ denotes the volume form of $M$. Since $\kappa<0$, the above relation is impossible.
\end{proof}

\section{Almost $\alpha$-cosymplectic manifolds }
In this section we study  an almost $\alpha$-cosymplectic manifold admitting quasi-Einstein structures. First we consider $V$ being collinear with Reeb vector field $\xi$.
\begin{theorem}
Let $(M,\phi,\xi,\eta,g)$ be a compact almost $\alpha$-cosymplectic manifold. Suppose that $M$ admits a quasi-Einstein structure $(g, V, m, \lambda)$ with $V=\eta(V)\xi$. If
$\alpha(3m+2\eta(V))\geq0$, then either $M$ is Einstein or $M$ is locally the
product of a K\"{a}hler manifold and an interval or unit circle $S^1$.
\end{theorem}
\begin{proof}
 As before we set $V=F\xi$ for some function $F$. By \eqref{2.5}, we have
\begin{equation}\label{5.1}
  \nabla_XV=X(F)\xi-F(\alpha\phi^2X+\phi hX).
\end{equation}
Using \eqref{5.1}, Formula \eqref{4.27} becomes
\begin{align*}
  Ric(X,Y)=&\lambda g(X,Y)-\frac{1}{2}(X(F)\eta(Y)+Y(F)\eta(X))\\
  &+Fg((\alpha\phi^2X+\phi hX),Y)+\frac{F^2}{m}\eta(X)\eta(Y).
\end{align*}
This is equivalent to
\begin{equation}\label{5.7}
  QX=\lambda X-\frac{1}{2}\Big(X(F)\xi+\eta(X)DF\Big)+F(\alpha\phi^2X+\phi hX)+\frac{F^2}{m}\eta(X)\xi.
\end{equation}
Differentiating \eqref{5.7} along $Y$ and using \eqref{2.5}, we conclude
\begin{align*}
 (\nabla_YQ)X=&-\frac{1}{2}\Big(g(X,\nabla_YDF)\xi+X(F)\nabla_Y\xi+g(\nabla_Y\xi,X)DF+\eta(X)\nabla_YDF\Big)\\
&+F\Big(\alpha(\nabla_Y\phi^2)X+(\nabla_Y\phi)hX+\phi(\nabla_Yh)X\Big)\nonumber\\
&+\frac{2F}{m}Y(F)\eta(X)\xi+\frac{F^2}{m}g(\nabla_Y\xi,X)\xi+\frac{F^2}{m}\eta(X)\nabla_Y\xi.\nonumber
\end{align*}
Contracting the pervious formula over $Y$ gives
\begin{align*}
 \frac{1}{2}\xi(r)=&-\frac{1}{2}\Big(\Delta F+2n\alpha\xi(F)+\xi(\xi(F))\Big)\\
&+F\Big(2n\alpha+{\rm trace}(h^2)\Big)+\frac{2F}{m}\xi(F)+\frac{2nF^2}{m}\alpha.\nonumber
\end{align*}

On the other hand, from \eqref{5.7} we have
\begin{equation*}
  r=(2n+1)\lambda-\xi(F)-2nF\alpha+\frac{F^2}{m}.
\end{equation*}
Inserting this into the foregoing relation gives
\begin{align}\label{6.30}
 0=-\frac{1}{2}\Delta F+F\Big(2n\alpha+{\rm trace}(h^2)\Big)+\frac{F}{m}\xi(F)+\frac{2nF^2}{m}\alpha.
\end{align}
Using \eqref{5.1} again and recalling \eqref{2.9}, we get ${\rm div}V=\xi(F)+2n\alpha F$. Thus
\begin{align}\label{6.31}
  {\rm div}(F^2V) =&V(F^2)+F^2{\rm div}V \\
  = & 3F^2\xi(F)+2n\alpha F^3\nonumber
\end{align}
Since $\Delta F^2=2F\Delta F+2\|DF\|^2$, multiplying \eqref{6.30} by $F$ and using \eqref{6.31} we give
\begin{align}\label{4.24}
 0=&-\frac{1}{2}\Delta F^2+\|DF\|^2+2F^2{\rm trace}(h^2)+\frac{2}{3m}{\rm div}(F^2V)+4nF^2\alpha(1+\frac{2F}{3m}).
\end{align}
Integrating this over $M$, we know
\begin{align*}
 0=\int_M\Big\{\|DF\|^2+2F^2{\rm trace}(h^2)+4nF^2\alpha\Big(1+\frac{2F}{3m}\Big)\Big\}dM.
\end{align*}

Under the assumption, we see that $F$ is constant. Furthermore, if $F=0$ it is obvious that $M$ is Einstein, and $h=0$ and $\alpha\Big(1+\frac{2F}{3m}\Big)=0$ if $F\neq0$.
Combining \eqref{6.31} with \eqref{4.24}, we find $\alpha=0$. Hence $M$ is a cosymplectic manifold. We complete the proof by Blair's result (cf.\cite{B}).
 \end{proof}

In the following we study the three dimensional case.
\begin{theorem}\label{T4.2}
Let $(M^3,\phi,\xi,\eta,g)$ be an $\alpha$-almost cosymplectic manifold. Suppose that $M$ admits a non-trivial quasi-Einstein structure $(g, V, m, \lambda)$ with $V=\eta(V)\xi$. Then  either $M$ is locally the
product of a K\"{a}hler manifold and an interval or unit circle $S^1$, or $V$ has constant length $m$. Moreover, if $M$ is compact, it is locally the
product of a K\"{a}hler manifold and an interval or unit circle $S^1$.
\end{theorem}
\begin{proof}
First it is well known that the curvature tensor of a 3-dimensional Riemannian manifold is given by
\begin{align}\label{5.4*}
  R(X,Y)Z=&g(Y,Z)QX-g(X,Z)QY+g(QY,Z)X-g(QX,Z)Y\\
           &-\frac{r}{2}\{g(Y,Z)X-g(X,Z)Y\}.\nonumber
  \end{align}
Hence substituting \eqref{5.7} into \eqref{5.4*} yields
\begin{align*}
  R(X,Y)Z=&g(Y,Z)\Big[-\frac{1}{2}\Big(X(F)\xi+\eta(X)DF\Big)+F(\alpha\phi^2X+\phi hX)+\frac{F^2}{m}\eta(X)\xi\Big]\\
&-g(X,Z)\Big[-\frac{1}{2}\Big(Y(F)\xi+\eta(Y)DF\Big)+F(\alpha\phi^2Y+\phi hY)+\frac{F^2}{m}\eta(Y)\xi\Big]\nonumber\\
&+g\Big(-\frac{1}{2}\Big(Y(F)\xi+\eta(Y)DF\Big)+F(\alpha\phi^2Y+\phi hY)+\frac{F^2}{m}\eta(Y)\xi,Z\Big)X\nonumber\\
&-g\Big(-\frac{1}{2}\Big(X(F)\xi+\eta(X)DF\Big)+F(\alpha\phi^2X+\phi hX)+\frac{F^2}{m}\eta(X)\xi,Z\Big)Y\nonumber\\
&+\Big(2\lambda-\frac{r}{2}\Big)\{g(Y,Z)X-g(X,Z)Y\}.\nonumber
  \end{align*}
Putting $Z=\xi$ yields
  \begin{align}\label{6.32}
  R(X,Y)\xi=&\frac{1}{2}Y(F)\phi^2X-\frac{1}{2}X(F)\phi^2Y+F\eta(Y)(\alpha\phi^2X+\phi hX)\\
  &+F\eta(X)(\alpha\phi^2Y+\phi hY)\nonumber\\
&+\Big(\frac{F^2}{m}-\frac{\xi(F)}{2}-F\alpha+2\lambda-\frac{r}{2}\Big)\{\eta(Y)X-\eta(X)Y\}.\nonumber
  \end{align}
Moreover, putting $Y=\xi$ in \eqref{6.32} we obtain from \eqref{2.6}
  \begin{equation*}
    \Big(\alpha^2+\frac{F^2}{m}-\xi(F)-F\alpha+2\lambda-\frac{r}{2}\Big)\phi^2 X=h^2X.
  \end{equation*}

Since the scalar curvature $r=3\lambda-\xi(F)-2F\alpha+\frac{F^2}{m}$, which is followed from \eqref{5.7}, inserting this into the pervious relation we have
  \begin{equation}\label{4.27*}
    \Big(\alpha^2+\frac{F^2}{2m}-\frac{1}{2}\xi(F)+\frac{1}{2}\lambda\Big)\phi^2 X=h^2X.
  \end{equation}

Now taking the inner product of \eqref{4.27*} with $\phi X$, we know $\phi h^2X=0$,  which implies $h=0$, i.e. $M$ is an $\alpha$-cosymplectic manifold. Moreover, Eq.\eqref{4.27*} implies
\begin{equation}\label{6.33}
  2\alpha^2+\frac{F^2}{m}-\xi(F)+\lambda=0.
\end{equation}

For an $\alpha$-cosymplectic manifold, the following formula holds (see \cite{OAM}):
\begin{equation*}
 R(X,Y)\xi=\alpha^2\{\eta(Y)X-\eta(X)Y\},
\end{equation*}
thus by comparing with \eqref{6.32} and replacing $X$ by $\phi X$ and $Y$ by $\phi Y$ respectively, we find
\begin{equation*}
\phi Y(F)\phi X=\phi X(F)\phi Y.
\end{equation*}
Moreover, letting $Y=DF$ gives $\phi X(F)\phi DF=0$, which implies $DF=\xi(F)\xi$. Therefore, using \eqref{6.33} we compute
\begin{equation*}
  \Delta F={\rm div}(DF)=\xi(\xi(F))+2\xi(F)\alpha=\frac{2F\xi(F)}{m}+ 2\xi(F)\alpha.
\end{equation*}
Inserting this and \eqref{6.33} into \eqref{6.30}, we get
\begin{align*}
 0=&\alpha\Big[-(\lambda+2\alpha^2)+2F+\frac{F^2}{m}\Big].\nonumber
\end{align*}
If $\alpha=0$, $M$ is cosymplectic. If $\alpha\neq0$, The above formula implies $-(\lambda+2\alpha^2)+2F+\frac{F^2}{m}=0$, that shows that $F$ is constant. Recalling \eqref{6.33} we have $F=-m$ for a non-trivial quasi-Einstein structure.
If $M$ is compact, $\alpha=0$ from \eqref{6.31}. Therefore we complete the proof.
\end{proof}
If $V$ is a conformal vector field, we obtain
\begin{theorem}
Let $(M^3,\phi,\xi,\eta)$ be an almost $\alpha$-cosymplectic manifold. Suppose that $M$ admits a quasi-Einstein structure $(g, V, m, \lambda)$ with $V$ being a conformal vector field. Then $V$ is Killing and $M$ is of constant scalar curvature.
\end{theorem}
\begin{proof}
As $V$ is a conformal vector field, we have $(\mathcal{L}_Vg)(X,Y)=2\rho g(X,Y)$ for any vector fields $X,Y$ and some function $\rho$ on $M$, hence \eqref{4.27} becomes
\begin{equation}\label{6.35}
  QX=(\lambda-\rho)X+\frac{1}{m} V^\flat(X) V.
\end{equation}
Substituting \eqref{6.35} into \eqref{5.4*} gives
\begin{align*}
  R(X,Y)Z=&\frac{1}{m}\Big[g(Y,Z)V^\flat(X)-g(X,Z) V^\flat(Y) \Big]V\\
&+\frac{1}{m}V^\flat(Y) g( V,Z)X-\frac{1}{m}V^\flat(X) g( V,Z)Y\\
&+\Big(2(\lambda-\rho)-\frac{r}{2}\Big)\{g(Y,Z)X-g(X,Z)Y\}.
  \end{align*}
Putting $Y=Z=\xi$ we have
\begin{align*}
  R(X,\xi)\xi=&\frac{1}{m}\Big[V^\flat(X)-\eta(X)\eta(V)\Big]V\\
&+\frac{1}{m}\eta(V)\eta(V)X-\frac{1}{m}V^\flat(X) \eta(V)\xi\nonumber\\
&+\Big(2(\lambda-\rho)-\frac{r}{2}\Big)\{X-\eta(X)\xi\}.\nonumber
  \end{align*}
Thus using \eqref{2.6} we obtain
\begin{align*}
&\frac{1}{m}\Big[V^\flat(X)-\eta(X)\eta(V)\Big]V+\frac{1}{m}\eta(V)\eta(V)X-\frac{1}{m}V^\flat(X) \eta(V)\xi\\
=&\frac{1}{m}V^\flat(\phi X)\phi V+\frac{1}{m}\eta(V)\eta(V)\phi X+2\Big(2(\lambda-\rho)-\frac{r}{2}+\alpha^2\Big)\phi^2 X-2h^2X.\nonumber
  \end{align*}
Now letting $X=V$ we conclude
\begin{equation*}
0=\frac{1}{m}\eta(V)\eta(V)\phi V+2\Big(\frac{\lambda-\rho}{2}+\alpha^2\Big)\phi^2 V-2h^2V.
\end{equation*}

Since $h^2-\alpha^2\phi^2=\frac{\mathrm{trace}(l)}{2}\phi^2$ (see \cite[Proposition 14]{OAM}), we know $V\in\mathcal{D}$, where $\mathcal{D}=\{X\in TM:\eta(X)=0\}.$ The conformal condition of $V$ implies $g(\nabla_\xi V,\xi)=\rho g(\xi,\xi)=\rho$, i.e. $\rho=-g(V,\nabla_\xi\xi)=0$.

Remark that the following formula holds (cf.\cite{GH2}):
\begin{equation}\label{6.38*}
\frac{2}{2n+3}X(r)+X(\rho)-\frac{2\rho}{m}V^\flat(X) = 0
\end{equation}
for $\mathrm{dim} M=2n+1.$
Therefore we see that the scalar curvature $r$ is constant.
\end{proof}
Finally, we intend to consider a three dimensional strictly $\alpha$-almost cosymplectic manifold (i.e. $h\neq0$), admitting a quasi-Einstein structure $(g,m,f,\lambda)$.
There exits a local orthonormal frame field $\mathcal{E}=\{e_1,e_2=\phi e_1,\xi\}$
such that $he_1=\mu e_1$ and $he_2=-\mu e_2$, where $\mu$ is a
positive non-vanishing smooth function of $M$. The following relation holds (\cite[Proposition 12]{OAM}):
\begin{align}
  &\nabla_\xi h=2ah\phi+\xi(\mu)s.\label{4.32}
\end{align}
Here $a$ is a function defined by $a=g(\nabla_\xi e_2,e_1)$ and $s$
is a $(1,1)$ tensor field defined by $se_1=e_1,se_2=-e_2$ and $s\xi=0$.
\begin{lemma}[\cite{OAM}]\label{L4.1}
With respect to $\mathcal{E}$ the Levi-Civita
connection $\nabla$ is given by
\begin{align*}
  &\nabla_\xi e=-a\phi e, \quad \nabla_\xi \phi e=ae,\;\;\nabla_\xi\xi=0, \\
  &\nabla_{e}\xi=\alpha e-\mu \phi e, \quad \nabla_{\phi e}\xi=-\mu e+\alpha\phi e,\\
&\nabla_{e}e=\frac{1}{2\mu}[(\phi e)(\mu)+\sigma(e)]\phi e-\alpha\xi, \quad \nabla_{\phi e}\phi e=\frac{1}{2\mu}[e(\mu)+\sigma(\phi e)]e-\alpha\xi,\\
&\nabla_{e}\phi e=-\frac{1}{2\mu}[(\phi e)(\mu)+\sigma(e)] e+\mu\xi,\quad \nabla_{\phi e}e=-\frac{1}{2\mu}[e(\mu)+\sigma(\phi e)]\phi e+\mu\xi,
\end{align*}
where $\sigma$ is the 1-form defined by $\sigma(\cdot)=Ric(\cdot ,\xi)$.
\end{lemma}
We say that (1,1)-type tensor field $\phi h$ on $(M,g)$ is said to be an\emph{ $\eta$-parallel}
tensor if it satisfies the equation
\begin{equation*}
g((\nabla_X\phi h)Y, Z) = 0
\end{equation*}
for all tangent vectors $X, Y, Z$ orthogonal to $\xi$ (see \cite{AYM}).
\begin{theorem}
There are no quasi-Einstein structures  $(g,m,f,\lambda)$ on a strictly  almost $\alpha$-cosymplectic manifold $(M^3,\phi,\xi,\eta)$ with $\phi h$ is $\eta$-parallel.
\end{theorem}
\begin{proof}
By \cite[Theorem 1]{AYM}, the Reeb vector field is an eigenvector field of the Ricci operator, then $\sigma(e)=\sigma(\phi e)=0$
and $Q\xi=-2(\mu^2+\alpha^2)\xi$.
Moreover, by \cite[Proposition 13]{AYM}, we know that $R(\phi X,\phi Y)\xi=0$ for any vector fields $X,Y$. In view of Lemma \ref{L3.2},
\begin{equation*}
g((\nabla_{\phi Y}Q)\phi X-(\nabla_{\phi X}Q)\phi Y,\xi)=0.
\end{equation*}
Using \eqref{2.5} we obtain
\begin{equation}\label{6.42}
Q\phi hY=\phi hQ Y
\end{equation}
for any vector filed $Y$ on $M$.

From \cite[Lemma 3]{OAM}, the Ricci operator may be expressed as
\begin{align*}
  QX=&\Big(\frac{1}{2}r+\alpha^2+\mu^2\Big)X+\Big(-\frac{1}{2}r-3\alpha^2-3\mu^2\Big)\eta(X)\xi\\
&+2\alpha \phi hX+ \phi(2ah\phi+\xi(\mu)s)X.
\end{align*}
Thus
\begin{align}
  Qe= & \Big(\frac{1}{2}r+\alpha^2+\mu^2+2\mu a\Big)e+(2\alpha\mu+ \xi(\mu))\phi e, \label{6.43*}\\
  Q\phi e =&  \Big(\frac{1}{2}r+\alpha^2+\mu^2-2\mu a\Big)\phi e+(2\alpha\mu+ \xi(\mu)) e.\label{6.44*}
\end{align}
Putting $Y=e$ in \eqref{6.42} and using the pervious equations, we obtain
\begin{align*}
&\mu \Big[\Big(\frac{1}{2}r+\alpha^2+\mu^2+2\mu a\Big)e+(2\alpha\mu+ \xi(\mu))\phi e\Big]\\
=&\mu\Big[\Big(\frac{1}{2}r+\alpha^2+\mu^2-2\mu a\Big)e+(2\alpha\mu+ \xi(\mu))\phi e\Big].
\end{align*}
This implies $a=0$ since $\mu\neq0$.
Since $\phi h$ is $\eta$-parallel,  we obtain from Lemma \ref{L4.1} that $e(\mu)=\phi e(\mu)=0$.

Differentiating \eqref{6.42} along any vector field $X$ gives
\begin{equation}\label{6.43}
(\nabla_XQ)\phi hY+Q(\nabla_X\phi h)Y=(\nabla_X\phi h)Q Y+\phi h(\nabla_XQ)Y.
\end{equation}
Letting $X=e$ and $Y=\xi$ and using \eqref{6.42} again yields
\begin{equation*}
0=-\alpha\phi he+h^2e.
\end{equation*}
Here we have used \eqref{2.5} and $e(\mu)=0$. This implies that $\mu=\alpha$ is constant.

Now taking $X=Y=e$ in \eqref{6.43} and using \eqref{6.43*}, \eqref{6.44*}, we find $r=0$. Thus by \eqref{3.10*} we obtain
\begin{equation*}
  -(m-1)|\mathrm{Ric}|^2=6\lambda^2,
\end{equation*}
which shows that $\lambda=0$ since $m\geq1$. Moreover, $\mathrm{Ric}=0$ if $m\neq1$. From \eqref{6.43*}, we see $\alpha=0$, that is impossible as $\alpha=\mu\neq0.$
Thus $m=1$.

Write $$Df=\xi(f)\xi+e(f)e+\phi e(f)\phi e.$$
By Lemma \ref{L3.2} and \eqref{2.10}, it follows
\begin{align*}
g(R(X,\xi)D f,\xi) =&g((\nabla_\xi Q)X-(\nabla_XQ)\xi,\xi)-4\alpha^2\{X(f)-\xi(f)\eta(X)\}\\
=&-4\alpha^2\{X(f)-\xi(f)\eta(X)\}\\
=&g(-2\alpha^2\phi^2X-2\alpha\phi hX,Df).
\end{align*}
Putting $X=e$ in this formula we derive
\begin{align*}
3e(f)=\phi e(f),\quad 3\phi e(f)=e(f).
\end{align*}
The above two formulas imply $e(f)=\phi e(f)=0$, thus $Df=\xi(f)\xi$. By the proof of Theorem \ref{T4.2}, we know $h=0$, which is a contradiction.
\end{proof}

\section{$K$-cosymplectic manifolds}
 Let $M$ be a $(2n+1)$-dimensional almost cosymplectic manifold defined in Section 2, namely the 1-form $\eta$ and the fundamental form $\omega$ are closed  and satisfy $\eta\wedge\omega^n\neq0$ at every point of $M$.

\begin{definition}[\cite{BG}]
An almost cosymplectic manifold $(M,\phi,\xi,\eta,g)$ is called a \emph{$K$-cosymplectic manifold} if the Reeb vector field $\xi$ is Killing.
\end{definition}

For a $K$-cosymplectic manifold $(M,\phi,\xi,\eta,g)$, by Theorem 3.11 in \cite{MNY} we know
\begin{equation*}
  \nabla\xi=\nabla\eta=0.
\end{equation*}
Moreover, it follows from  Theorem 3.29 in \cite{MNY} that
\begin{equation}\label{6.1}
 R(X,Y)\xi=0\quad\text{for all}\;X,Y\in\mathfrak{X}(M).
\end{equation}
That shows that $Q\xi=0$.

As $V^\flat$ is closed, Equation \eqref{4.27} is equivalent to
\begin{equation}\label{5.21}
  \nabla_YV=\lambda Y-QY+\frac{1}{m}g(V,Y)V.
\end{equation}
Via this formula one derives easily
\begin{align}\label{5.22}
  R(X,Y)V =& \nabla_X\nabla_YV-\nabla_Y\nabla_XV-\nabla_{[X,Y]}V \\
  = & (\nabla_YQ)X-(\nabla_XQ)Y+\frac{1}{m}[V^\flat(X)QY\nonumber\\
  &-V^\flat(Y)QX]+\frac{\lambda}{m}[V^\flat(Y)X-V^\flat(X)Y].\nonumber
\end{align}
By \eqref{6.1} and $Q\xi=0$, taking an inner product of \eqref{5.22} with $\xi$ gives
\begin{align*}
  \frac{\lambda}{m}[V^\flat(Y)\eta(X)-V^\flat(X)\eta(Y)]=0.
\end{align*}
This implies that either $\lambda=0$ or $V=\eta(V)\xi$.

As before we set $\eta(V)=F$.
Since $V=F\xi$ and $\nabla\xi=0$,  \eqref{5.21} becomes
\begin{equation}\label{5.25}
  Y(F)\xi=\lambda Y-QY+\frac{F^2}{m}\eta(Y)\xi.
\end{equation}
Due to $Q\xi=0$, taking  $Y=\xi$ implies
\begin{equation}\label{5.24*}
  \xi(F)=\lambda+\frac{F^2}{m}.
\end{equation}

On the other hand, contracting \eqref{5.25} over $Y$ we also have
\begin{equation*}
 \xi(F)=(2n+1)\lambda-r+\frac{F^2}{m},
\end{equation*}
which, combining with \eqref{5.24*}, yields $r=2n\lambda$ is constant. Further, from \eqref{5.25} we know
\begin{equation*}
  QY=\lambda (Y-\eta(Y)\xi).
\end{equation*}
That is to say that $M$ is an $\eta$-Einstein manifold.

Summing up the above discussion, we actually proved the following conclusion.
\begin{theorem}
Let $(M,\phi,\xi,\eta,g)$ be a $(2n+1)$-dimensional K-cosymplectic manifold. Suppose that $M$ admits a closed, non-steady quasi-Einstein structure $(g, V, m, \lambda)$.
Then $M$ is $\eta$-Einstein .
\end{theorem}
\begin{corollary}
Let $(M,\phi,\xi,\eta,g)$ be a $(2n+1)$-dimensional K-cosymplectic manifold. Suppose that $M$ admits a non-steady quasi-Einstein structure $(g, f, m, \lambda)$.
Then $M$ is $\eta$-Einstein .
\end{corollary}
For $V$ being a conformal vector field, we also have
\begin{theorem}
Let $(M^{2n+1},\phi,\xi,\eta)$ be a $K$-cosymplectic manifold. Suppose that $M$ admits a quasi-Einstein structure $(g, V, m, \lambda)$ with $V$ being a conformal vector field. Then $V$ is Killing and $M$ is of constant scalar curvature. Moreover, if the quasi-Einstein structure is non-steady, $M$ is $\eta$-Einstein.
\end{theorem}
\begin{proof}
Since $Q\xi=0$, it follows from \eqref{6.35} that
\begin{equation}\label{5.46}
 0=Q\xi=(\lambda-\rho)\xi+\frac{1}{m} \eta(V)V.
\end{equation}
This shows that either $V\in\mathbb{R}\xi$ or $V\in\mathcal{D}$.

If $V\in\mathbb{R}\xi$, we write $\eta(V)=F$ and $V=F\xi$. Differentiating this along $X$ gives $\nabla_XV=X(F)\xi$ since $\nabla\xi=0$.
Because $V$ is a conformal vector field, we get $X(F)\eta(Y)+Y(F)\eta(X)=2\rho g(X,Y)$ for any $X,Y$. Now replacing $X$ and $Y$ by $\phi X$ and $\phi Y$, respectively, we
obtain easily $\rho=0$. For $V\in\mathcal{D}$, it is easy to get $\rho=0$ from the conformal condition of $V$. Recalling \eqref{6.38*}, we thus know that $r$ is constant.

Moreover, if $\lambda\neq0$, it implies from \eqref{5.46} that $\eta(V)\neq0$, that is, $V\in\mathbb{R}\xi$. We complete the proof by \eqref{6.35}.
\end{proof}

\end{document}